\documentclass[a4paper]{amsart}

\usepackage{amssymb,mathrsfs,enumitem,amsmath,bbold}
\usepackage[mathscr]{euscript}
\usepackage{color}
\definecolor{Chocolat}{rgb}{0.36, 0.2, 0.09}
\definecolor{BleuTresFonce}{rgb}{0.215, 0.215, 0.36}
\definecolor{EgyptianBlue}{rgb}{0.06, 0.2, 0.65}

\usepackage[colorlinks,final,hyperindex]{hyperref}
\hypersetup{citecolor=BleuTresFonce, urlcolor=EgyptianBlue, linkcolor=Chocolat}

\usepackage{fourier}

\usepackage[ps,all]{xy}
\xyoption{line}
\CompileMatrices

\newtheorem{theorem}{Theorem}[section]

\newtheorem{lemma}[theorem]{Lemma}
\newtheorem{proposition}[theorem]{Proposition}

\theoremstyle{definition}

\newtheorem{remark}[theorem]{Remark}
\newtheorem{definition}[theorem]{Definition}

\DeclareMathAlphabet{\pazocal}{OMS}{zplm}{m}{n}

\def\calO{\pazocal{O}}

\def\calT{\pazocal{T}}

\def\calX{\pazocal{X}}

\def\bbW{\mathbb{W}}

\DeclareMathOperator{\id}{id}

\DeclareMathOperator{\Set}{\mathsf{Set}}

\DeclareMathOperator{\Ass}{Ass}

\DeclareMathAlphabet{\mathbbold}{U}{bbold}{m}{n}

\def\k{\mathbbold{k}}

\newcommand{\LYRY}[5]{\ensuremath{
 \xygraph{
!{<0pt,0pt>;<4pt,0pt>:<0pt,-4pt>::}
!{(1,4)}*{\scriptscriptstyle #1}
!{(1,3)}="a"
!{(1,1)}="b"
!{(-1,-1)}="c"
!{(3,-1)}="d"
!{(-2,-4)}*{\scriptscriptstyle #2}
!{(-2,-3)}="e"
!{(0,-4)}*{\scriptscriptstyle #3}
!{(0,-3)}="f"
!{(2,-3)}="g"
!{(4,-3)}="h"
!{(2,-4)}*{\scriptscriptstyle #4}
!{(4,-4)}*{\scriptscriptstyle #5}
"b"-"c"
"b"-"d"
"c"-"e"
"c"-"f"
"d"-"g"
"d"-"h"
}
}}

\newcommand{\lbincomb}[5]{\ensuremath{%
\vcenter{\hbox{\xymatrix@R=.4pc@C=.2pc{%
			#3\ar@{-}[dr] &&#4\ar@{-}[dl] &\\
			&*+[o][F-]{#2}\ar@{-}[dr]&&#5\ar@{-}[dl]\\
			&&*+[o][F-]{#1}\ar@{-}[d]&\\
			&&*{}&
}}}}}

\begin{document}

\title{Word operads and admissible orderings}

\author{Vladimir Dotsenko}
\address{School of Mathematics, Trinity College, Dublin 2, Ireland}
\email{vdots@maths.tcd.ie}

\begin{abstract}
We use Giraudo's construction of combinatorial operads from monoids to offer a conceptual explanation of the origins of Hoffbeck's path sequences of shuffle trees, and use it to define new monomial orders of shuffle trees. One such order is utilised to exhibit a quadratic Gr\"obner basis of the Poisson operad.  
\end{abstract}

\maketitle

\section*{Introduction} 

In \cite{DK}, the notion of a shuffle operad was introduced and utilised to develop a formalism of operadic Gr\"obner bases. The latter is indispensable for purposes of linear, homological, and homotopical algebra for operads. In order to use operadic Gr\"obner bases, one has to come up, for each specific application, with a monomial order that extracts the ``correct'' leading terms from the defining relations. In the decade that elapsed since dissemination of \cite{DK}, most applications of shuffle operads have been using the path-lexicographic order introduced in that paper, or its minor variations. The purpose of this short note is to offer a conceptual explanation of the origins of that order which also leads to a plethora of new orders which have remained unnoticed until now. In particular, we demonstrate how one of such orders can be used to exhibit a quadratic Gr\"obner basis of the Poisson operad; a construction of that sort is required as an intermediate step in one of the arguments in a recent preprint \cite{KW}.

The main observation at the heart of this note is that the combinatorics of path sequences in free shuffle operads \cite{BD,DK,Ho} can be naturally derived from the the construction of operads from monoids due to Giraudo \cite{Gir} (related to previous work of M\'endez and Nava \cite{MN} and also Berger and Moerdijk \cite{BM}). Our definition of an order on the monoid of ``quantum monomials'' appears to be new; besides the application we present, monomial orders based on this monoid can be used to prove freeness of certain operadic modules, leading to functorial PBW theorems for various universal enveloping algebras~\cite{D19,DT}.

The word operad construction below applies to either of the three commonly used types of operads: symmetric, nonsymmetric, and shuffle. We refer to symmetric operads as operads, while the two other types of operads always appear with a specific adjective. The reader is invited to consult~\cite{BD,LV} for background information on operad theory. Most of our constructions utilise operads in the symmetric monoidal category $(\Set,\times)$; in the only situation when one has to consider $\k$-linear operads, we state it explicitly.  We denote by $\mu_\tau(-;-,\ldots,-)$ the structure maps of a given operad (here $\tau$ is a 2-level tree, and the type of the tree is prescribed by the type of the operad that we consider, i.e. symmetric, nonsymmetric, or shuffle). 

\textbf{Acknowledgements. } The author is grateful to Anton Khoroshkin and Pedro Tamaroff for useful discussions. 

\section{Word operads}\label{monoids}

The following definition is essentially due to Giraudo~\cite{Gir}; we use the language of species \cite{BLL} for clarity. 

\begin{definition}[Word operad]
Suppose that $(M,\star)$ is a monoid. The species $\bbW_M$ is defined by the formula $\bbW_M(I)=M^I$. 
For each map $f\colon I\to\underline{n}$, we have a map
 \[
\gamma_f\colon\bbW_M(\underline{n})\times\bbW_M(f^{-1}(1))\times\cdots\times\bbW_M(f^{-1}(n))\to\bbW_M(I)
 \]
defined by the formula
 \[
\gamma_f(a;b_1,\ldots,b_n)(i):=a(f(i))\star b_{f(i)}(i).
 \] 
By a direct inspection, these maps satisfy the properties required of compositions in an operad. The resulting operad is called the \emph{word operad of~$M$}.
We can also consider the associated \emph{shuffle word operad} $\bbW_M^{sh}$, and the associated \emph{nonsymmetric word operad} $\bbW_M^{ns}$.
\end{definition} 

The two crucial combinatorial objects associated to monomials in free shuffle operads are path sequences and permutation sequences \cite{BD,DK}. Let us explain how those arise naturally in the context of word operads.

\begin{definition}
Let $\calX$ be a sequence of sets with $\calX(0)=\varnothing$. We denote 
 \[
\underline{\calX}:=\bigsqcup\limits_{n\ge 1} \calX(n) , 
 \]
the union of all these sets taken together. We may consider the free shuffle operad $\calT^{sh}(\calX)$
and the free monoid $T(\underline{\calX})$. The map of operads 
  \[
\theta\colon\calT^{sh}(\calX)\to\bbW^{sh}_{T(\underline{\calX})}
  \]
is the unique morphism of shuffle operads extending the sequence of maps \[\theta_n\colon\calX(n)\to T(\underline{\calX})^n\] with $(\theta_n(x))_k=x$ for all $1\le k\le n$. For an element $T\in\calT(\calX)$, the element $\theta(T)$ is called the \emph{path sequence of $T$}. 
 \end{definition}

By a direct inspection, for an element $T\in\calT(\calX)(n)$, the sequence $\theta_n(T)$ coincides with the path sequence of a tree tensor defined by Hoffbeck \cite{Ho}, see also \cite{BD,DK}; for example,
 \[
\theta_3\left( \lbincomb{a}{b}{1}{3}{2} \right)=(ab,a,ab) .
 \] 
Our set-up somewhat clarifies the key feature of Hoffbeck's construction, informally expressed by the statement ``path sequences of tree tensors behave well under operadic compositions''.

\begin{definition}
Let $\calX$ be a sequence of sets with $\calX(0)=\varnothing$.  
The map of operads 
  \[
\sigma\colon\calT^{sh}(\calX)\to\Ass^{sh}
  \]
is the unique morphism of shuffle operads extending the sequence of maps \[\sigma_n\colon\calX(n)\to \Ass(n)\] with $(\sigma_n(x))=\id$. For an element $T\in\calT(\calX)$, the element $\sigma(T)$ is called the \emph{permutation of $T$}. 
 \end{definition}

By a direct inspection, for an element $T\in\calT(\calX)(n)$, the element $\sigma_n(T)$ coincides with the permutation sequence of a tree monomial defined in~\cite{DK}; for example,
 \[
\sigma_n\left( \lbincomb{a}{b}{1}{3}{2} \right)=\begin{pmatrix}1&2&3\\ 1&3&2\end{pmatrix} .
 \] 

\begin{remark}
Neither the map $\sigma$ nor the map $\theta$ are equivariant with respect to the symmetric group actions: they only make sense in the universe of shuffle operads. 
\end{remark}

The reason permutations of shuffle trees are useful is that the map from the free shuffle operad into the Hadamard product of the operads $\bbW_{T(\underline{\calX})}$ and $\Ass$ is injective~\cite{BD,DK}. This allows one to reduce the more intricate combinatorics of trees to various familiar features of words and permutations.  

\medskip

To state the main result of this section, we need the definition of an ordered shuffle operad.

\begin{definition}[Ordered shuffle operad] 
A shuffle operad $\calO$ is said to be \emph{ordered}, if each component $\Gamma(n)$ is equipped with a partial order $\prec$ for which every structure map $\mu_T$ is an increasing function of its arguments: if we replace one of the arguments of any structure map $\mu_T(-;-,\ldots,-)$ by an element from the same component of $\calO$ which is greater with respect to $\prec$, the result is also greater with respect to $\prec$. 
\end{definition}

For instance, the two-level tree 
 \[
\LYRY{}{1}{3}{2}{4}
 \]
represents the composite $\alpha(\beta_1(a_1,a_3),\beta_2(a_2,a_4))$ in a shuffle operad. One of the implications of the above definition is that if $\beta_1\prec\beta_1'$, then we must have 
 \[
\alpha(\beta_1(a_1,a_3),\beta_2(a_2,a_4))\prec\alpha(\beta_1'(a_1,a_3),\beta_2(a_2,a_4)) .
 \]

A particular case of an ordered set operad is an ordered monoid: an ordered monoid is an ordered operad concentrated in arity one. More classically, one can say that an ordered monoid is a monoid $(M,\star)$ equipped with a partial order $\prec$ for which $a\prec a'$ implies $a\star b\prec a'\star b$ and $b\star a\prec b\star a'$. It turns out that the $\bbW$-construction satisfies the following remarkable property. 

\begin{proposition}
Suppose that $M$ is an ordered monoid. Then $\bbW_M^{sh}$ with the lexicographic order of tuples is an ordered shuffle operad.  
\end{proposition}

\begin{proof}
The proof of a very particular case of this result (where both the set operad and the monoid are free) given by Hoffbeck \cite[Prop.~3.5]{Ho} works verbatim in full generality. Conveniently, even the terminology used in that proof ``word sequence'' suits our formalism perfectly. 
\end{proof}

\section{Quantum monomials and the Poisson operad}

The toy example we consider in this section is not very deep, but it indicates a possible universe of applications of word operads. Namely, we shall use word operads to show that the Poisson operad has a quadratic Gr\"obner basis. Of course, the Poisson operad is one of the most famous operads ever considered, and both obvious applications of Gr\"obner bases (determining normal forms and proving Koszulness) do not give anything new for it. However, much more complicated operads \cite{D19,DT} can be studied by similar methods; also, results of \cite[Sec.~3.3]{KW} substantially rely on a version of this result. 

Recall that the Poisson operad $\mathsf{Pois}$ is generated by a symmetric binary operation $a_1,a_2\mapsto a_1\cdot a_2$ and a skew-symmetric binary operation $a_1,a_2\mapsto \{a_1,a_2\}$ which satisfy the identities
\begin{gather*}
(a_1\cdot a_2)\cdot a_3=a_1\cdot(a_2\cdot a_3),\\
\{a_1,a_2\cdot a_3\}=\{a_1,a_2\}\cdot a_3+\{a_1,a_3\}\cdot a_2,\\
\{a_1,\{a_2,a_3\}\}=\{\{a_1,a_2\},a_3\}-\{\{a_1,a_3\},a_2\}.
\end{gather*}
It is well known that the free Poisson algebra on a vector space $V$ is isomorphic to $S(\mathsf{Lie}(V))$, which leads to a convenient choice of normal forms in the Poisson operad: it has a basis made of commutative associative products (made of the operation $a_1,a_2\mapsto a_1\cdot a_2$) of Lie monomials (made of the operation $a_1,a_2\mapsto \{a_1,a_2\}$). However, detecting those normal forms on the level of Gr\"obner bases for operads is a tricky task. To explain why it is the case, let us consider the second relation, the ``Leibniz rule'' relating the two operations. In the associated shuffle operad, we have three relations arising from that one:
\begin{gather*}
\{a_1,a_2\cdot a_3\}=\{a_1,a_2\}\cdot a_3+\{a_1,a_3\}\cdot a_2,\\
-\{a_1\cdot a_3,a_2\}=-\{a_1,a_2\}\cdot a_3+a_1\cdot\{a_2,a_3\},\\
-\{a_1\cdot a_2,a_3\}=-a_1\cdot\{a_2,a_3\}-\{a_1,a_3\}\cdot a_2.
\end{gather*}
If we were to find a Gr\"obner basis leading to the normal forms mentioned above, each of these relations must have its left-hand side as the leading term. Known orderings of monomials in the free operad fail to accomplish that, yet it is possible to find such an ordering. This is exactly where we shall use word operads.

\begin{theorem}
There exists an ordering of shuffle tree monomials in two binary generators $\mu$ and $\lambda$ (encoding our two binary operations $\mu\colon a_1,a_2\mapsto a_1\cdot a_2$ and $\lambda\colon a_1,a_2\mapsto \{a_1,a_2\}$) for which the left-hand sides of the three Leibniz rules above are the leading monomials. For that ordering, the defining relations of the Poisson operad form a quadratic Gr\"obner basis.
\end{theorem}

\begin{proof}
Let us consider the monoid of ``quantum monomials''$\mathsf{QM}=\langle x,y,q\rangle/(xq-qx,yq-qy,yx-xyq)$. It is immediate to see that each element of that monoid has a unique representative of the form $x^ky^lq^m$ where $k,l,m\ge 0$. 

We define an order on these representatives by putting $x^ky^lq^m\prec x^{k'}y^{l'}q^{m'}$ if $k>k'$ or $k=k'$ and $l<l'$, or $k=k'$ and $l=l'$ and $m<m'$. (Note the ``counterintuitive'' comparison $k>k'$.) 

\begin{lemma}
This order makes $\mathsf{QM}$ into an ordered monoid. 
\end{lemma}

\begin{proof}
We should show that for any $a,b,c\in \mathsf{QM}$, whenever $a\prec b$, we have $ac\prec bc$ and $ca\prec cb$. Let $a=x^ky^lq^m$, $b=x^{k'}y^{l'}q^{m'}$, $c=x^{k''}y^{l''}q^{m''}$. Note that we have 
\begin{gather*}
ac=x^{k+k''}y^{l+l''}q^{m+m''+lk''},\\
bc=x^{k'+k''}y^{l'+l''}q^{m'+m''+l'k''},\\
ca=x^{k+k''}y^{l+l''}q^{m+m''+kl''},\\
cb=x^{k'+k''}y^{l'+l''}q^{m'+m''+k'l''}.
\end{gather*}
Thus, if $a\prec b$ because $k>k'$, we have $ac\prec bc$ and $ca\prec cb$, as $k+k''>k'+k''$. If $a\prec b$ because $k=k'$ and $l<l'$, we have $ac\prec bc$ and $ca\prec cb$, as $k+k''=k'+k''$ and $l+l''<l'+l''$. Finally, if $a\prec b$ because $k=k'$ and $l=l'$ but $m<m'$, we have $ac\prec bc$ and $ca\prec cb$, as $k+k''=k'+k''$, $l+l''=l'+l''$, while $m+m''+lk''<m'+m''+l'k''$ and $m+m''+kl''<m'+m''+k'l''$.
\end{proof} 

We now consider the map $\psi$ from $\calT^{sh}(\mu,\lambda)$ into the word operad $\bbW^{sh}_{\mathsf{QM}}$ defined as follows: 
 \[
\psi(\lambda):=(y,y), \quad \psi(\mu):=(x,x).
 \]
This makes $\calT^{sh}(\mu,\lambda)$ into an ordered operad: to compare two shuffle trees $T_1$ and $T_2$, we compare $\psi(T_1)$ and $\psi(T_2)$ in $\bbW^{sh}_{\mathsf{QM}}$. Let us extend this partial order to a full monomial order arbitrarily, e.g. via a superposition with the path-lexicographic order. We note that the left-hand sides of the three Leibniz rules above are the leading monomials. Indeed, this follows from the fact that each of the elements $\{x,xy\}$ is smaller than each of the elements $\{y,yx\}$ in $\mathsf{QM}$. 

To prove that we obtain a Gr\"obner basis, we note that the associativity relations for $\mu$ form a Gr\"obner basis for any ordering, and so does the Jacobi identity for $\lambda$. By our choice of leading terms of the Leibniz rules, we already ensure that commutative associative products of Lie monomials are normal forms. Thus, no further elements can possibly belong to the reduced Gr\"obner basis, since that would create extra linear dependencies between the normal forms.
\end{proof}

\bibliographystyle{amsplain}
\providecommand{\bysame}{\leavevmode\hbox to3em{\hrulefill}\thinspace}

\end{document}